\documentclass{amsart}
\usepackage{amsfonts}
\usepackage{amsmath}
\usepackage{amssymb}
\usepackage{graphicx}
\usepackage{hyperref,color}%
\usepackage{comment}
\providecommand{\U}[1]{\protect\rule{.1in}{.1in}}

\newcommand{\rem}[1]{}
\newtheorem{theorem}{Theorem}
\theoremstyle{plain}

\newtheorem{definition}[theorem]{Definition}

\newtheorem{example}[theorem]{Example}

\newtheorem{lemma}[theorem]{Lemma}

\newtheorem{remark}[theorem]{Remarks}

\numberwithin{equation}{section}
\def\l{\lambda}
\def\M{\mathcal{M}}
\def\MS{\mathcal{MS}}
\def\GL{\mathcal{GL}}

\begin{document}
\title[Simultaneous diagonalization via congruence]{Solving the problem of simultaneous diagonalization of complex symmetric matrices via congruence}
\author{Miguel D. Bustamante}
\address{School of Mathematics and Statistics, University College Dublin, Dublin 4, Ireland}
\email{miguel.bustamante@ucd.ie}
\author{Pauline Mellon}
\address{School of Mathematics and Statistics, University College Dublin, Dublin 4, Ireland}
\email{pauline.mellon@ucd.ie}
\author{M. Victoria Velasco}
\address{Departamento de An\'{a}lisis Matem\'{a}tico, Facultad de Ciencias, Universidad
de Granada, \indent 18071 Granada, Spain}
\email{vvelasco@ugr.es}
\thanks{This work was partially supported by the Spanish Project MTM2016-76327-C3-2-P (AEI/FEDER,UE)}
\thanks{This work was also supported by the award of the Distinguished Visitor Grant of the School of Mathematics and Statistics, University College Dublin to the third author}
\date{09 July 2020}

\begin{abstract}
 We provide a solution to the problem of simultaneous
\emph{diagonalization via congruence} of a
given set of $m$  complex symmetric $n\times n$ matrices $\{A_{1},\ldots,A_{m}\}$,
 by showing that it
can be reduced to a possibly lower-dimensional problem where the question is rephrased
in terms of the classical problem of simultaneous \emph{diagonalization via
similarity} of a new related set of matrices. We provide a procedure to determine in a finite number of steps whether or not a set of matrices is simultaneously diagonalizable by congruence. This solves a long standing problem in the complex case.

\begin{center}
\textbf{This is a preprint of our work published in:\\
SIAM J. Matrix Anal. Appl., 41(4), 1616--1629 (October 2020).\\
\url{https://doi.org/10.1137/19M1280430}}
\end{center}

\end{abstract}

\maketitle

\section{Introduction}

The aim of this paper is to characterise when a given set of $n\times n$ complex
symmetric matrices, $A_{1},\ldots,A_{m}$, are simultaneously diagonalizable via
congruence  (SDC), namely, when there exists a non-singular $n\times n$ complex matrix $P$ such that 
$$P^TA_iP\ \hbox{ is diagonal, for all}\ 1\leq i\leq m.$$ 

Since Weierstrass in 1868
\cite{We} gave
sufficient conditions for the simultaneous diagonalization by congruence of two real symmetric matrices, several authors  \cite{Hi-U, Hi-Ur-Tor, Ji-Li} have extended those results and there has been applications of the real results to areas as diverse as
quadratic programming \cite{YGR, Yi, Be}, variational analysis \cite{Hi-Ma}, signal processing \cite{Li,Pham, Ti-Ye} and medical imaging analysis \cite{As, Car, So-Co, Wan}, among others. In the case of complex matrices, Hong, Horn and Johnson laid the framework in the 1980s {for the particular case of unitary transformations} by proving \cite{H-H-J1,H-H-J2} that
 there is a unitary $U$ satisfying $U^TA_iU$ is diagonal for all $i$ if, and only if, the set $\{A_i\overline{A_j}:1\leq i,j\leq m\}$ is a commuting family.  In {a more general} case they solved the problem for pairs of complex symmetric (or Hermitian) matrices with {the restriction that at least one of them be} non-singular.  We provide a solution in the general case {of complex symmetric matrices} by
translating it into a simpler problem, at a possibly reduced dimension, regarding
simultaneous diagonalizability by similarity of a new set of related matrices. 
We do this using the concept of matrix pencils so that the general problem is reduced to (possibly) lower dimensions  \emph{a priori} by calculating the intersection of the kernels of the matrices $A_1, \ldots, A_m$. Once this is done, reduced $r\times r$ matrices ($r\leq n$) $\widetilde{A}_i$ (Lemma \ref{kernels2} below) can be dealt with in a more standard manner, thanks to the existence of a non-singular matrix pencil. This allows us to obtain fairly simple necessary and sufficient conditions for SDC in Theorem \ref{main}  below.

The authors were initially motivated to tackle SDC for complex symmetric matrices by a problem that arose naturally in the area of evolution algebras.

We recall here that an evolution algebra is defined as a commutative
algebra $A$ {}{over $\mathbb{C}$} for which there exists a basis $\widetilde{B}=\{\widetilde{e}_{i}%
:i\in\Lambda\}$ such that $\widetilde{e}_{i} \widetilde{e}_{j} =0$, for every
$i,j\in\Lambda$ with $i\neq j$. In other words, the multiplication table of
$A$ relative to $\widetilde{B}$ is diagonal. Such a basis is called
natural. Evolution algebras were introduced in \cite{Tian} and \cite{Ti-Vo} in
the study of non-Mendelian genetics and are, in
general, not associative. The problem concerning the authors was to determine when a given algebra $A$ is an evolution algebra. In other words, if $B$ is a basis of $A$ and
the multiplication table of $A$ with respect to $B$ is not diagonal, we established the conditions under which there exists a natural basis $\widetilde{B}$ of $A$, giving $A$ the structure of an evolution algebra.  If
$B=\{e_{1},...,e_{n}\}$ and%
\begin{equation}
\label{eq:multiplication_rule}
e_{i}e_{j}={\sum\limits_{k=1}^{n}}m_{ijk}\,e_{k}\,,\quad i,j=1,\ldots,n,
\end{equation}
we define the \emph{structure matrices} of $A$ with respect to $B$ as the $n\times n$
matrices $M_{k}(B)=(m_{ijk})_{1\leq i,j\leq n},$ for $k=1,\ldots,n.$ {}{Notice that the  structure matrices $M_{k}(B)$ are symmetric and complex because $A$ is commutative and its base field is $\mathbb{C}$.} A main result in \cite{Bu-Me-Ve1} proves that $A$ is an evolution algebra if, and only if,
{}{the complex symmetric matrices} $M_{1}(B), \ldots ,M_{n}(B)$ are simultaneously diagonalizable via congruence.

{There are other areas where the complex results might be applied. One of the most important applications is in the area of signal processing, in particular in the classical problem of blind source separation \cite{belouchrani1997blind, yeredor2000blind, yeredor2002non}. In its simplest form, and appropriate to our notation, this problem amounts to finding a nonsingular complex matrix $Q$ relating $n$ sets of measurements (denoted by the complex random vector $x$ of dimension $n$) and $n$ statistically independent, \emph{but unknown}, sources (denoted by the complex random vector $s$ of dimension $n$), via the linear relation $x = Q^* s$, where $Q^*$ denotes the conjugate transpose of $Q$. {To investigate how this relates to the SDC problem} we discuss a method introduced in \cite{yeredor2000blind}. Consider a generalised second characteristic function, defined in terms of the $x$ variables as
$$\psi_x(\tau) := \ln E[\exp(\tau^T \overline{x})], \qquad \tau \in \mathbb{C}^n\,,$$
where bar denotes complex conjugation and $E$ denotes the expectation. In terms of the $s$ variables, this reads
$$\psi_x(\tau) = \ln E[\exp((Q\tau)^T \overline{s})] =: \psi_s(\mu)\,, \qquad \mu = Q \tau\,.$$
Let us consider $m$ so-called ``processing points'' $\tau^{(1)}, \ldots, \tau^{(m)}$. We now define the following complex symmetric matrices $A_1, \ldots, A_m$ by their components:
$$(A_j)_{k\ell} := \left.\frac{\partial^2 \psi_x(\tau)}{\partial \tau_k \partial\tau_\ell}\right|_{\tau = \tau^{(j)}} , \qquad k, \ell = 1, \ldots, n, \qquad j = 1, \ldots, m.$$ 
It is then easy to show 
$$A_j = Q^T D_j Q\,,\qquad (D_j)_{k\ell} := \left.\frac{\partial^2 \psi_s(\mu)}{\partial \mu_k \partial\mu_\ell}\right|_{\mu = P\tau^{(j)}} , \,\,\,\quad k, \ell = 1, \ldots, n, \quad j = 1, \ldots, m,$$
where the matrices  $D_j$ are diagonal, due to the statistical independence of the components of $s$. This is, of course, the SDC problem, and its solution provides the complex matrix $Q$ that allows one to unveil the unknown independent sources starting from an arbitrary set of measurements. In real-life applications, experimental or numerical errors will lead to matrices $A_1, \ldots, A_m$  that are not exactly SDC, so ``approximate joint diagonalization'' is the correct concept, which consists of the variational problem of finding a complex nonsingular matrix $P$ such that $P^T A_j P$ is as diagonal as possible, in some metric (see, for example, \cite{belouchrani1997blind, co-di, yeredor2002non}).

In summary, the solution to the SDC problem provided extends earlier 
results from the 1980s, solves the initial motivating question for the authors related to evolution algebras and may have 
an impact on applications in optimisation or signal
processing as described above.}

 In Section \ref{sec:2} we provide notation and definitions. In Section \ref{sec:3} we solve the SDC problem in the case of complex symmetric matrices and present a finite step procedure to determine whether a given set of matrices is SDC or not. In Section \ref{sec:4} we discuss possible avenues of further research. 
\section{Notation}
\label{sec:2}

Let $\mathcal{M}_{n,m}$ denote all $n\times m$ matrices over $\mathbb{C}$. Let $\mathcal{M}_{n}:=\mathcal{M}_{n,n}$, let
$\mathcal{MS}_{n}$ be all symmetric elements in $\mathcal{M}_{n}$ and let
$\mathcal{GL}_{n}$ be all invertible elements in $\mathcal{M}_{n}$. A diagonal matrix in $\M_n$ with diagonal entries $d_1, \ldots, d_n$ will be written as $D=\mathrm{\mathrm{diag}}\big(d_1,\ldots, d_n\big)$. For $A \in \M_n$ we denote its $ij$ component by $A_{ij}$ or $\big(A\big)_{ij}$ and the zero and identity element in $\mathcal{M}_{n}$ are denoted $0_n$ and $I_n$, respectively.
We recall that $A \in \M_n$ is said to be orthogonal if $A^T=A^{-1}$ (where $A^T$ denotes the usual transpose of $A$) and is said to be unitary if $\bar{A}^T = A^{-1}$ (where $\bar{A}$ denotes the entrywise complex conjugate of $A$); matrices $A,B \in \M_n$ are said to be congruent if there exists $P \in \GL_n$ such that $P^TAP=B$ and are said to be similar if there exists $P \in \GL_n$ such that $P^{-1}AP=B$. Congruent (or similar) matrices have the same rank. In fact, $A$ and $B$ are congruent if, and only if, they have the same rank \cite[Theorem 4.5.12]{Horn}, hence $A$ similar to $B$ implies $A$ congruent to $B$ but the converse does not hold. We
introduce the following definitions for a set of matrices in $\mathcal{M}_{n}$.

\begin{definition}
\label{def:1}
Let $A_{1},\ldots,A_{m} \in \M_n$. We say $A_{1},\ldots,A_{m}$ are simultaneously
diagonalizable via congruence (SDC for short) if 
there exists $P\in \GL_{n}$ and diagonal matrices $D_{1},\ldots,D_{m} \in \mathcal{M}_{n}
$ such that %
\[
P^{T}A_{j}P=D_{j},\quad j=1,\ldots,m\,.
\]
Of course, if $A_{1},\ldots,A_{m}$ are SDC then they are necessarily symmetric.

\end{definition}

\begin{definition}
\label{def:2}
Let $L_{1},\ldots,L_{m} \in \M_n$. We say
$L_{1},\ldots, L_{m}$ are simultaneously diagonalizable via
similarity (SDS for short) if there exists $P\in
\GL_{n}$  and diagonal matrices $D_{1},\ldots,D_{m} \in \M_n$ such that %
\[
P^{-1}L_{j}P=D_{j},\quad j=1,\ldots,m\,.
\]

\end{definition}

It is important to remark that even when  $A_1, \ldots, A_m$ in  Definition \ref{def:1} or  $L_1, \ldots, L_m$ in  Definition \ref{def:2} are real, the resulting matrices $P$ and $D_j$ may have to be complex, as illustrated in Example \ref{complexneeded} below.  The following result is well known (see for instance \cite[Theorems 1.3.12 and 1.3.21]{Horn})
and  means that SDS is easy to check in practice, in contrast to SDC.

\begin{theorem}
\label{well}  Let $L_1, \ldots, L_m \in \mathcal{M}_n$. These matrices are simultaneously diagonalizable by similarity (SDS) if, and only if, they are all diagonalizable  by similarity and they pairwise commute. 
\end{theorem}

\section{Solving the SDC\ problem}
\label{sec:3}
Let $S^{2m-1}:=\{x\in\mathbb{C}%
^{m}:\Vert x\Vert=1\}$, where $\left\Vert \cdot\right\Vert $ denotes the usual
Euclidean norm. We use the standard concepts of linear pencil and maximum pencil rank.

\begin{definition}
Let $A_{1}, \ldots, A_{m} \in \mathcal{M}_{n}$.
Define the associated linear pencil to be the map
\[
A: \mathbb{C}^{m}\longrightarrow \mathcal{M}_{n}\ \hbox{ by }\ A(\lambda) = \sum_{j=1}^{m}
\lambda_{j} A_{j},\ \hbox{where}\  \lambda = 
\begin{pmatrix}\lambda_1\\ \vdots \\\lambda_m\end{pmatrix} 
 \in \mathbb{C}^m.
\]

\end{definition}

Since $\mathrm{rank}A\left(  \lambda\right)  =\mathrm{rank}A\left(  \frac{\lambda}%
{\|\lambda\|}\right)  $, for $\lambda\neq0$, it follows that
\[
\sup_{\substack{\lambda\in\mathbb{C}^{m}}} \mathrm{rank}{A}(\lambda)=\sup
_{\substack{\lambda\in S^{2m-1}}} \mathrm{rank}{A}(\lambda).
\]

In addition, since $\{ \mathrm{rank}{A}(\lambda):\lambda \in  S^{2m-1}\} \subseteq \{0,1,\ldots,n\}$, it follows that the above supremum must be achieved. In other words, there exists some $\lambda_0 \in S^{2m-1}$ such that 
$$\sup_{\substack{\lambda\in\mathbb{C}^{m}}}\mathrm{rank}{A}(\lambda)=\sup
_{\substack{\lambda\in S^{2m-1}}} \mathrm{rank}{A}(\lambda)=\mathrm{rank}{A}(\lambda_0).$$

\begin{definition}
Let $A_{1},\ldots,A_{m} \in \mathcal{M}_{n}$.
The rank of the associated linear pencil is $r:=\sup_{\substack{\lambda\in\mathbb{C}^{m}}}\mathrm{rank}{A}(\lambda).$ We refer to $r$ as the (maximum pencil) rank of $A_{1}%
,\ldots,A_{m}$  and denote it as $r=\mathrm{rank}(A_{1},\ldots,A_{m}).$ In the case that $r=n$
we say that the pencil is non-singular.
From above, $r=\mathrm{rank}(A_{1},\ldots,A_{m})= \mathrm{rank}{A}(\lambda_0)$, for some $\lambda_{0}\in S^{2m-1}$.
\end{definition} 

The following simple lemma is important.

\begin{lemma}\label{intersections}
Let $A_{1},\ldots,A_{m} \in \mathcal{M}_{n}$ and let $r=\mathrm{rank}(A_{1},\ldots,A_{m})=\mathrm{rank} A(\lambda_{0})$, for some $\lambda_{0}\in S^{2m-1}$. Then
$$\dim\big(\bigcap_{j=1}^m \mathrm{ker} A_j\big)=n-r\ \text{ if, and only if,}\ 
\bigcap_{j=1}^m \mathrm{ker} A_j= \mathrm{ker} {A}(\lambda_0).$$
\end{lemma}
\begin{proof}
Clearly $\bigcap_{j=1}^m \mathrm{ker} A_j\subseteq \mathrm{ker} {A}(\lambda)$ and hence 
\begin{equation}\label{kerdim}
\dim\big(\bigcap_{j=1}^m \mathrm{ker} A_j\big) \leq n-\mathrm{rank}A(\l),\  \text{for all}\ \l  \in \mathbb{C}^m.
\end{equation}
In particular, for maximum pencil rank $r=\mathrm{rank} A(\l_0)$ we have 
\newline $\bigcap_{j=1}^m \mathrm{ker} A_j\subseteq \mathrm{ker} {A}(\lambda_0)$, \,and \,$\dim(\mathrm{ker} {A}(\lambda_0))=n-r$ then gives the result.
\end{proof}

We will see later that  $\dim\big(\bigcap_{j=1}^m \mathrm{ker} A_j\big)=n-r$ is necessary for 
$A_{1},\ldots,A_{m}$ to be SDC and, in this case, it follows  from  the above that the subspace $\mathrm{ker} {A}(\lambda_0)$ is actually independent of the point $\lambda_0$ satisfying $r=\mathrm{rank} A(\lambda_{0})$.

~\\
\subsection{The SDC problem for $n\times n$ matrices with non-singular pencils}
~\\

We now solve the SDC problem for symmetric matrices $A_{1},\ldots,A_{m}\in \mathcal{M}_{n}$, in the particular case that $\mathrm{rank}(A_{1}%
,\ldots,A_{m})=n$. The proof follows ideas from \cite[Theorem 4.5.17]{Horn},  \cite[Lemma 1]{Ji-Li} and \cite[p.230]{Uhlig}.
 In particular, the simple observation that if $A(\lambda)$ is invertible then
\begin{equation*}
(P^{T}A(\lambda)P)(P^{-1}A(\lambda)^{-1}A_{j}P)=P^{T}A_{j}P, \, \hbox{for}\  j= 1,\ldots, m, \ \hbox{and any}\ P \in \mathcal{GL}_n
\end{equation*} 
motivates our first main result and proves it in the obvious direction.

\begin{theorem}
\label{mainfullrank}Let ${A}_{1},\ldots,{A}_{m} \in \mathcal{MS}_{n}$ have maximum pencil rank $n$.
For any $\lambda_0 \in \mathbb{C}^{m}$ with
$
\mathrm{rank}A(\lambda_0)=
n$ then  
 $$ A_{1}, \ldots, A_m\  \hbox{are SDC if, and only if,}\  
A(\lambda_0)^{-1}A_{1}, \ldots, A(\lambda_0)^{-1}A_{m}\ \hbox{ are SDS}.$$

\end{theorem}

\begin{proof} Let $\lambda_0 \in \mathbb{C}^{m}$ satisfy
$\mathrm{rank}A(\lambda_0)=
n$.
\newline
In the forward direction, we assume that $A_{1}, \ldots, A_m$ are SDC and 
let $P \in \mathcal{GL}_n$ satisfy $P^T A_j P$ is diagonal, for $j=1, \ldots, m$. Then
$\ P^{T}A(\lambda_0)P$  is diagonal and invertible giving
\[
P^{-1}A(\lambda_0)^{-1}A_{j}P=(P^{T}A(\lambda_0)P)^{-1}\big(P^{T}A_{j}P\big)
\]
is diagonal, for $1\leq j\leq m$ and we are done.

In the opposite direction, assume that  $A(\lambda_0)^{-1}A_{1}, \ldots, A(\lambda_0)^{-1}A_{m}$ are SDS and let $P \in \mathcal{GL}_n$ satisfy $D^{(j)} := P^{-1}A(\lambda_0)^{-1}A_{j}P$ is
diagonal, for $j=1, \ldots, m$. 
We define symmetric matrices $B_j  := P^{T}A_{j}P$ and $B(\lambda_0) := P^{T}A(\lambda_0) P$, to give
\begin{equation}
\label{eq:B_j=BD}
B_j = B(\lambda_0)D^{(j)} ,\quad j=1, \ldots, m.
\end{equation}
Taking the transpose of this latter equation implies $ B(\lambda_0)$ commutes with $D^{(j)}$, for $j=1, \ldots, m$. Component-wise, this means that for all $1\leq k,l \leq n,$ 
$$\big(B(\lambda_0)\big)_{k \ell}\big(D^{(j)}\big)_{\ell \ell} = \big(D^{(j)}\big)_{kk}\big(B(\lambda_0)\big)_{k \ell}, \  \text{for all} \ j=1, \ldots, m.$$
In particular, for all $1\leq k,l \leq n,$
\begin{equation}
\label{eq:off-diag}
\big(B(\lambda_0)\big)_{k \ell} = 0 \quad \text{if}\ \ \big(D^{(j)}\big)_{kk} \neq \big(D^{(j)}\big)_{\ell \ell}\  \text{for any}\  j=1, \ldots, m.
\end{equation}


Write $D^{(j)}=\mathrm{diag}(\alpha^j_1,\ldots,\alpha^j_n)$, for $1\leq j\leq m$, and let $p_j$ satisfy $$\alpha^j_1=\ldots=\alpha^j_{p_j}\neq \alpha^j_{p_j +1}$$ ($p_j$ is the length of the first run of identical diagonals in $D^{(j)}$) and define $n_1:=\min_{1\leq j\leq m} p_j.$ Define $\alpha^{(j)}_1:=\alpha^j_1$
and $\alpha^{(j)}_2:=\alpha^j_{n_1 +1}$ so that 
$$\alpha^{(j)}_1=\alpha^j_1=\ldots =\alpha^j_{n_1},\ \hbox{for all}\ \  1\leq j\leq m,$$ and we may write 
$$ D^{(j)}=\alpha_1^{(j)} I_{n_1} \oplus  \mathrm{diag}(\alpha^{(j)}_2,\alpha^j_{n_1 +2},\ldots,\alpha^j_n),\ \hbox{for all}\ \  1\leq j\leq m,$$ and there is some $j\in \{1,\ldots, m\}$ for which $\alpha^{(j)}_1 \neq \alpha^{(j)}_2$ ($I_{n}$ denotes the $n\times n$ identity matrix).
We repeat a similar process twice more on  $\mathrm{diag}(\alpha^{(j)}_2,\alpha^j_{n_1 +2},\ldots,\alpha^j_n)$ to find $n_2$, $n_3$ and 
$ \alpha^{(j)}_3:=\alpha^j_{n_1 + n_2 +1}$ so that, $\hbox{for all}\ 1\leq j\leq m$,
$$D^{(j)}=\alpha_1^{(j)} I_{n_1} \oplus  \alpha_2^{(j)} I_{n_2}\oplus \alpha_3^{(j)} I_{n_3}\oplus \mathrm{diag}(\alpha^j_{n_1 +n_2+n_3+1}, \ldots,\alpha^j_n),
$$ while there is some $j$ for which $\alpha^{(j)}_1 \neq \alpha^{(j)}_2$ and some $k$ for which $\alpha^{(k)}_2 \neq \alpha^{(k)}_3$.
\smallskip

If now $\alpha^{(j)}_1 = \alpha^{(j)}_3$,  for all $1\leq j\leq m$, then we may re-order the diagonal entries to amalgamate 
$\alpha_1^{(j)}I_{n_1}$ and $\alpha_3^{(j)} I_{n_3}$, namely, there is an orthogonal permutation matrix  $R \in \mathcal{GL}_n$ with $$ R^TD^{(j)}R=R^{-1}D^{(j)}R=\alpha_1^{(j)} 
I_{n_1+n_3} \oplus  \alpha_2^{(j)} I_{n_2}\oplus \mathrm{diag}(\alpha^j_{n_1 +n_2+n_3+1}, \ldots,\alpha^j_n),$$ $\hbox{ for all}\ 1\leq j\leq m.$

We continue this process of finding $\alpha_l^{(j)}$s for $\mathrm{diag}(\alpha^j_{n_1 +n_2+n_3+1}, \ldots,\alpha^j_n)$. We amalgamate, as described above, where necessary so that 
for some orthogonal $U \in \M_n$ and  $\text{ all} \ 1\leq j\leq m$ we have 
\begin{equation}
\label{block diag1}U^TD^{(j)}U = \alpha_1^{(j)} I_{p_1} \oplus \ldots \oplus\alpha_d^{(j)} I_{p_d}\end{equation}
$\  \text{subject to the condition that for}\  1\leq a<b\leq d,\ \text{there is some}\  j \in \{1,\ldots,m\}$ {with} $\alpha_a^{(j)} \neq \alpha_b^{(j)}.$ Of course, $d \leq n$ and $d$ is as small as possible satisfying the above.

We now write $U^TB(\lambda_0)U$ as a $d\times d$ block matrix, whose $(a,b)$ sub-block, denoted here $[U^TB(\lambda_0)U]_{ab},$ is of size $n_a\times n_b$, for $1\leq a,b\leq d$. Then $U^TB(\lambda_0)U$ commutes with $ U^TD^{(j)}U$, for all $1\leq j\leq m$, since  $B(\lambda_0)$ commutes with $D^{(j)}$, for all $1\leq j\leq m$, and $U$ is orthogonal.  This commutativity then yields a block version of (\ref{eq:off-diag}).

Specifically, since
$$[U^TD^{(j)}U]_{aa}=\alpha_a^{(j)} I_{n_a}$$
 and we have from (\ref{block diag1})
that if $a < b$ then $\alpha_a^{(j)} \neq \alpha_b^{(j)}$, we get 
\begin{equation}
\label{block version}
[U^TB(\lambda_0)U]_{ab} = 0\ \text{if}\ \ a \neq b.
\end{equation}

 In other words, we have a block diagonal decomposition 
\begin{equation}
\label{block diag2}
U^TB(\lambda_0)U = C_1 \oplus \ldots \oplus C_d\,,
\end{equation}
where $C_a \in \mathcal{MS}_{n_a},\,\, a=1, \ldots, d$. 
\newline
As each $C_a$ must be symmetric, we can diagonalize it via a unitary transformation, as in \cite[Corollary 2.6.6 (a)]{Horn}.
In other words, for each $a=1, \ldots, d$ there exists $V_a \in \mathcal{GL}_{n_a}$ unitary and $D_a$ a non-negative diagonal matrix such that
\begin{equation}
\label{block diag3}
V^T_a C_aV_a = D_a \,,\ \ 1\leq a \leq d.\end{equation}
We recall also  that the diagonal entries of $D_a$ are the singular values of $C_a.$

Defining 
$V := V_1 \oplus \ldots \oplus V_d$ and $D := D_1 \oplus \ldots \oplus D_d$ (diagonal) then (\ref{block diag2}) and  (\ref{block diag3}) give
\begin{equation}
\label{block diag4}
V^T\big(U^TB(\lambda_0)U\big)V=D.
\end{equation}

Defining now $Q = PUV$, and since $B(\lambda_0)=P^TA(\lambda_0)P$,  (\ref{block diag4}) implies
$Q^{T} A(\lambda_0) Q = D$. 
In addition,
\begin{align*}
Q^T A_j Q &= V^TU^T(P^T A_j P)UV = V^{T}(U^TB_jU)V\\
&=V^{T}(U^TB(\lambda_0)D^{(j)}U)V,\  \ \hbox{from  (\ref{eq:B_j=BD}) }\\
&=V^{T}(U^TB(\lambda_0)U)(U^TD^{(j)}U)V, \ \hbox{since $U$ is orthogonal}\\
&=\big(V^{T}(U^TB(\lambda_0)U)V\big)(U^TD^{(j)}U),\  \hbox{as $V$ commutes with $U^TD^{(j)}U$ by (\ref{block diag1})}\\
&=D(U^TD^{(j)}U),\  \ \hbox{from  (\ref{block diag4}) }\\
&=D\big(\alpha_1^{(j)} I_{p_1} \oplus \ldots \oplus\alpha_d^{(j)} I_{p_d}\big), \ \hbox{ from  (\ref{block diag1})}
\end{align*}
 which is clearly diagonal for all $j=1, \ldots, m$.
\end{proof}

We recall from Theorem \ref{well} that $A(\lambda_0)^{-1}A_{1}, \ldots, A(\lambda_0)^{-1}A_{m}$ are SDS if, and only if, they are all diagonalizable by similarity and they pairwise commute. It follows  now from Theorem \ref{mainfullrank} that property SDS of the matrices $A(\lambda_0)^{-1}A_{1}, \ldots, A(\lambda_0)^{-1}A_{m}$ is independent of the particular $\lambda_0$ chosen.

~\\

\subsection{The SDC problem for $n\times n$ matrices with arbitrary pencil rank}

\subsubsection{Preliminaries: Diagonal matrices.}

\begin{lemma}
\label{lem:diagonal1}
Let $D_1, \dots,D_m$ be diagonal  matrices in $\M_n$, ${D}$ be the associated linear pencil and $r$ be its maximum pencil rank.
Then the following hold:
\begin{enumerate}
\item[(i)]$D_1, \ldots, D_m$ have zeros in the same $(n-r)$ diagonal positions and 
$$\dim\big(\bigcap_{j=1}^m \mathrm{ker} D_j\big)=n-r.$$
\item[(ii)] There is an orthogonal $Q \in \M_n$ such that $Q^TD_jQ=\widetilde{D}_j\oplus 0_{n-r}$, where $\widetilde{D}_j\in \M_r$ is diagonal, $1\leq j\leq m$.
 \end{enumerate}
\smallskip

Moreover the pencil  $\widetilde{D}$ associated to matrices $\widetilde{D}_1,\ldots,\widetilde{D}_m \in M_r$ is non-singular
\newline (and if $\lambda_0$ satisfies $r=\mathrm{rank} D(\l_0)$ then $\widetilde{D} (\l_0) \in \GL_r$).

\end{lemma}
\begin{proof} 

Since $D$ has maximum pencil rank $r$, we choose $\l_0 \in S^{2m-1}$ with $r= \mathrm{rank}{D}(\lambda_0).$
Writing $D_j=\mathrm{\mathrm{diag}}\big(d^j_1,\ldots, d^j_n\big) \in \M_n$, $1\leq j\leq m$, we define vectors
$u_i=\begin{pmatrix}d^1_i\\\vdots\\d^m_i\end{pmatrix} \in \mathbb{C}^m$, for $1\leq i\leq n.$
By direct calculation $${D}(\lambda)=\mathrm{\mathrm{diag}}\big(\lambda\cdot u_1, \ldots, \lambda\cdot u_n\big),\ \hbox{for all}\  \lambda = 
\begin{pmatrix}\lambda_1\\ \vdots \\\lambda_m\end{pmatrix} 
 \in \mathbb{C}^m,$$ where $\cdot$ represents the dot product on $\mathbb{C}^m$ given by 
$\begin{pmatrix}z_1\\ \vdots \\z_m\end{pmatrix} \cdot \begin{pmatrix}w_1\\ \vdots \\w_m\end{pmatrix}= \sum_{i=1}^mz_iw_i.$

Since  $r=\mathrm{rank}{D}(\lambda_0),$ we can then assume without loss of generality (up to rearrangement of the basis vectors)  that $\lambda_0\cdot u_i\neq 0\,, \,\,\, 1\leq i \leq r$ and $\lambda_0 \cdot u_j = 0\,, \,\,\, r+1\leq j \leq n$. In particular, $u_i\neq 0$, for $1\leq i\leq r.$
\newline Define $h:\mathbb{C}^m \longrightarrow \mathbb{C}$ by $h(\lambda) = \Pi_{i=1}^r \lambda\cdot  u_i$. Since $h$ is continuous, the set $A:=h^{-1}(\mathbb{C}\setminus \{0\})$ is open in $\mathbb{C}^m$ and since $\lambda_0 \in A$, we have that for some $s>0$, $\lambda_0+v \in A$ and hence $h(\lambda_0+v) \neq 0$, for all $v \in \mathbb{C}^m, \,\|v\| < s$.
 This gives $\mathrm{rank}{D}(\lambda_0 + v) \geq r$ and since $r$ is the maximum rank of ${D}(\lambda)$, it follows that $\mathrm{rank}{D}(\lambda_0 + v)=r$ and therefore $(\lambda_0+ v) \cdot u_j = 0$, for all $j$ with $r+1 \leq j \leq n$. Thus $v \cdot u_j= 0$, for all $ r+1 \leq j \leq n$, and all $v \in \mathbb{C}^m, \,\|v\| < s$. This is impossible unless $u_j=0$,  for all $r+1 \leq j \leq n$ (otherwise $v=\overline u_j(\frac{s}{2\|u_j\|})$ will give a contradiction). In other words
$$0=(u_j)_k=d^k_j=(D_k)_{jj}, \ \hbox{for all}\ r+1\leq j \leq n \ \hbox{and all}\ 1\leq k \leq m$$
namely, $ D_1, \ldots, D_m$ have zeros in the same $n-r$ diagonal positions. It follows that 
$\dim\big(\bigcap_{j=1}^m \mathrm{ker} D_j\big)\geq n-r.$ On the other hand, Lemma \ref{intersections} and (\ref{kerdim}) then imply $\dim\big(\bigcap_{j=1}^m \mathrm{ker} D_j\big)=n-r$ and $ \bigcap_{j=1}^m \mathrm{ker} D_j= \mathrm{ker} {D}(\lambda_0).$

 (ii) For $\l_0 \in S^{2m-1}$ with $r= \mathrm{rank}{D}(\lambda_0)$, we see in the proof of (i) that there is an orthogonal (permutation) matrix $Q \in \M_n$ satisfying  $Q^TD_jQ=\widetilde{D}_j\oplus 0_{n-r}$, where $\widetilde{D}_j\in \M_r$ is diagonal, $1\leq j\leq m$. Let $\widetilde D$ be the reduced linear pencil associated to $\widetilde{D}_1,\ldots,\widetilde{D}_m.$ Then  $Q^TD(\l_0)Q=\widetilde D(\l_0)\oplus 0_{n-r}$, so $\mathrm{rank}\widetilde D(\l_0)=\mathrm{rank}D(\l_0)=r$ and $\widetilde D(\l_0) \in \GL_r.$
\end{proof}

The next theorem enables us, when considering whether or not a set of $n\times n$ matrices is SDC, to reduce the problem to a set of $r\times r$ matrices, where $r$ is the maximum pencil rank.

\begin{theorem}\label{reducerank}
  Let $A_1, \ldots, A_m \in \MS_n$  have maximum pencil rank $r$.
Then $$ A_1, \ldots, A_m\ \text{are SDC if, and only if}\ \dim\big(\bigcap_{j=1}^m \mathrm{ker} A_j\big)=n-r\ \text{and there exists}\  P \in \GL_n $$ 
$$\ \text{with}\ \ P^TA_jP=\widetilde{D}_j\oplus 0_{n-r},\ \hbox{where}\  \widetilde{D}_j\in \M_r\ \hbox{is diagonal},\ \ \ 1\leq j\leq m.$$ 
\smallskip

Moreover, if either of the above conditions is satisfied,
the pencil  $\widetilde{D}$ associated to matrices $\widetilde{D}_1,\ldots,\widetilde{D}_m \in M_r$ is non-singular
(and if $\lambda_0$ satisfies $r=\mathrm{rank} A(\l_0)$ then $\widetilde{D} (\l_0) \in \GL_r$).
\end{theorem}
\begin{proof}  Let $A_1, \ldots, A_m \in \MS_n$  have maximum pencil rank $r$. Choose any $\l_0 \in S^{2m-1}$ satisfying $r=\mathrm{rank} A(\l_0)$.

For the forward direction, assume that  $A_1, \ldots, A_m$ are SDC. Then there exists  $S \in \GL_n$ and diagonal  matrices $D_1, \ldots, D_m$ such that 
\begin{equation}\label{above}S^T A_j S = D_j, \quad  j=1, \ldots, m\,.\end{equation}
Let $D$ be the pencil associated to matrices $D_1,\ldots,D_m$.
Then $S^T {A}(\lambda)S = {D}(\lambda),$ for all $\l \in \mathbb{C}^m,$ 
so maximum pencil ranks for $A(\l)$ and $D(\l)$ agree and 
$r=\mathrm{rank} A(\l_0)=\mathrm{rank} D(\l_0).$ From Lemma \ref{lem:diagonal1} (ii) there then exists 
 an orthogonal $Q \in \M_n$ such that $Q^TD_jQ=\widetilde{D}_j\oplus 0_{n-r}$, where $\widetilde{D}_j\in \M_r$ is diagonal, $1\leq j\leq m$ and  $\widetilde{D}(\lambda_0)\in \GL_r$. Then $P=SQ$ gives 
$$P^TA_jP=\widetilde{D}_j\oplus 0_{n-r},$$ for $ 1\leq j\leq m$ as desired and for pencil $\widetilde{D}$ associated to matrices $\widetilde{D}_1,\ldots,\widetilde{D}_m$ clearly $r=\mathrm{rank} \widetilde{D}(\l_0)$. 

Since
$S^T {A}(\lambda_0)S = {D}(\lambda_0)$ we have 
 $\mathrm{ker} {A}(\lambda_0) =S(\mathrm{ker} {D}(\lambda_0))$. On the other hand, from Lemma \ref{lem:diagonal1} (i), Lemma \ref{intersections} and (\ref{above}) we have  $S(\mathrm{ker}{D}(\lambda_0))=S(\bigcap_{j=1}^m {\ker} D_j )=\bigcap_{j=1}^m {\ker} A_j$. In other words 
$\bigcap_{j=1}^m {\ker} A_j=\mathrm{ker} {A}(\lambda_0)$ has dimension $n-r$.
The opposite direction is trivial.\\
\end{proof}

\subsubsection{The general case of non-diagonal matrices with arbitrary pencil rank}~\\
~

The following Lemma holds regardless of diagonalizability and is key to solving the SDC problem in the general case.

\begin{lemma}\label{kernels2}
	 Let $A_1, \ldots, A_m \in \MS_n$  have maximum pencil rank $r.$ Then  $$\dim(\bigcap_{j=1}^m \mathrm{ker} A_j)=n-r\ \hbox{ if, and only if,}$$ there exists $Q \in \GL_n$ with  
\begin{equation}\label{dagger2}
Q^TA_jQ=\widetilde{A}_j\oplus 0_{n-r},\ \hbox{where}\  \widetilde{A}_j\in \MS_r,\ \ 1\leq j\leq m.
\end{equation}
\smallskip

Moreover, if either of the above conditions is satisfied, the pencil $\widetilde{A}$ associated to matrices $\widetilde{A}_1,\ldots,\widetilde{A}_m \in M_r$ is non-singular
(and if $\lambda_0$ satisfies $r=\mathrm{rank} A(\l_0)$ then $\widetilde{A} (\l_0) \in \GL_r$).
\end{lemma}

\begin{proof}
 Let $A_1, \ldots, A_m \in \MS_n$  have maximum pencil rank $r$.  

In the forward direction, assume that $\mathcal{V}:=\bigcap_{j=1}^m \mathrm{ker} A_j$ has $\dim(\mathcal{V})=n-r.$ 
Choose a basis $v_{r+1}, \ldots, v_n$ of $\mathcal{V}$ and extend by vectors $v_1, \ldots,v_r$ to get a  basis $v_1, \ldots, v_n$  of $\mathbb{C}^n.$ Let $Q\in \GL_n$ be the matrix whose $i$th column is given by the vector $v_i$. 
For $r+1\leq i\leq n$, we have $v_i\in \ker A_j$ and hence  $Q^TA_jQ(e_i)=Q^T A_j(v_i)=0$, for all $1\leq j\leq m$ (and $e_i$ is the column vector with $1$ in the $i$th position and all other entires $0$).
In other words, columns $r+1$ to $n$ of $Q^TA_jQ$ are identically zero and, since $Q^TA_jQ$ is symmetric, it follows that 
\begin{equation}\label{arbA}Q^TA_jQ=\widetilde{A}_j\oplus 0_{n-r},\ \hbox{where}\  \widetilde{A}_j\in \MS_r,\ 1\leq j\leq m\end{equation} as desired.
In the opposite direction, assume that  (\ref{dagger2}) holds for some $Q \in \GL_n$. Then 
$Q\big(0_r \oplus \mathbb{C}^{n-r}\big)\subseteq \bigcap_{j=1}^m \mathrm{ker} A_j$ so $\dim(\bigcap_{j=1}^m \mathrm{ker} A_j)\geq n-r$. Equality now follows from (\ref{kerdim}).

Finally, if the conditions in the statement hold,  the reduced pencil $\widetilde{A}$ associated with 
$\widetilde{A}_1,\ldots, \widetilde{A}_m \in M_r$ has maximum pencil rank $r$, since for any $\l_0 \in \mathbb{C}^m$ with $r=\mathrm{rank} A(\l_0)$ then (\ref{arbA}) implies
that $\mathrm{rank} \widetilde{A}(\l_0)=r$ and $\widetilde{A}(\l_0)\in \GL_r.$
\end{proof}
\begin{remark} 
We note that by choosing the basis vectors $v_1, \ldots, v_n$ in the above proof  to be orthogonal, with respect to the complex inner product $<z,w>:=z\cdot \overline w$ on $\mathbb{C}^n$, $Q$ can be chosen to be unitary.
\end{remark}

Lemma \ref{kernels2} therefore allows us to find  matrices $ \widetilde{A}_1,\ldots,  \widetilde{A}_m$ satisfying (\ref{dagger2}) using only the kernels of the $A_j$.  This enables us, subject to the condition that
$\dim\big(\bigcap_{j=1}^m \mathrm{ker} A_j\big)=n-r,$  to reduce the dimension of the problem by proving that 
 $$A_1, \ldots, A_m\ \text{are SDC} \ \hbox{in}\ M_n  \ \hbox{if, and only if,}\   \widetilde{A}_1,\ldots,  \widetilde{A}_m
\ \text{are SDC} \ \hbox{in}\ M_r.$$Theorem \ref{mainfullrank} then motivates 
 the following definition.

\begin{definition}[Reduced maximal-rank matrices]\label{reducedrank} Let $A_1, \ldots, A_m \in \MS_n$ have maximum pencil rank $r$ and satisfy $\dim(\bigcap_{j=1}^m \mathrm{ker} A_j)=n-r$. Let $ \widetilde{A}_1,\ldots,  \widetilde{A}_m$ be as in 
(\ref{dagger2}) and fix  $\lambda_0\in S^{2m-1}$ with $r=rank A(\l_0)$. Reduced pencil $ \widetilde{A}$ then has  $\widetilde{A}(\l_0)\in \GL_r$.

We define the $r\times r$ matrices
\begin{equation}
\label{eq:Ldef}
L_j \big(=L_j(\lambda_0)\big):={\widetilde{A}(\lambda_0)}^{-1} \widetilde{A}_j, \quad 1\leq j\leq m.
\end{equation}

\end{definition}

\begin{remark}
 $L_1, \ldots, L_m$ are not symmetric in general and $\sum_{j=1}^m (\lambda_0)_j L_j = I_n$.
In addition, Theorem \ref{mainfullrank} states that $L_1, \ldots, L_m$ are SDS if, and only if, $ \widetilde{A}_1,\ldots,  \widetilde{A}_m$ are SDC and consequently the condition is independent of the particular
$\lambda_0$ chosen in the definition. For this reason we write $L_j$ instead of $L_j(\lambda_0)$ . 
\end{remark}

The following is our main theorem.

\begin{theorem}\label{main}
 Let $A_1, \ldots, A_m \in \MS_n$ have maximum pencil rank $r$. Then $$A_1, \ldots, A_m\ \text{are SDC if, and only if,}\ \dim(\bigcap_{j=1}^m \mathrm{ker} A_j)=n-r\ \text{and}\ 
L_1, \ldots, L_m\ \text{ are SDS}$$
where $L_1, \ldots, L_m$ are as in Definition \ref{reducedrank} above.
\end{theorem}
\begin{proof}
Let $A_1, \ldots, A_m \in \MS_n$ have maximum pencil rank $r$ and choose $\lambda_0\in S^{2m-1}$ satisfying $r=rank A(\l_0)$.

In the forward direction, assume now that $A_1, \ldots, A_m$ are SDC.

 From Theorem \ref{reducerank} $\ \dim(\bigcap_{j=1}^m \mathrm{ker} A_j)=n-r$ and 
there exists $P \in \GL_n$ such that 
\begin{equation}\label{one}
P^TA_jP=\widetilde{D}_j\oplus 0_{n-r},\ \hbox{where}\  \widetilde{D}_j\in \M_r\ \hbox{is diagonal},\ \ \ 1\leq j\leq m.
\end{equation} 
In addition, if  $\widetilde{D}$ is the pencil associated to $r\times r$ matrices $\widetilde{D}_1,\ldots,\widetilde{D}_m$ then $\widetilde{D}(\lambda_0) \in \GL_r$.
Lemma \ref{kernels2} then gives $Q \in \GL_n$ with 
\begin{equation}\label{two}Q^TA_jQ=\widetilde{A}_j\oplus 0_{n-r},\ \hbox{where}\  \widetilde{A}_j\in \M_r,\ \ \  1\leq\  j\leq m
\end{equation}
and 
$\widetilde{A}(\lambda_0)\in \GL_r$, for reduced pencil 
$\widetilde{A}(\lambda_0) = \sum_{j=1}^m (\lambda_{0})_j \widetilde{A}_j$.

Thus for $R=Q^{-1}P$ (\ref{one}) and (\ref{two}) give
\begin{equation}\label{three}R^T\left(\widetilde{A}_j\oplus 0_{n-r}\right)R=\widetilde{D}_j\oplus 0_{n-r},\  1\leq\  j\leq m.\end{equation}
Writing $R$ as a block matrix, 
$$R = \left(
\begin{array}
[c]{cc}%
S & T\\
U &  V
\end{array}
\right),$$
for $S \in \M_r, V\in \M_{n-r}, U \in \M_{n-r,r}, T \in \M_{r,n-r}$, it follows from (\ref{three}) and matrix multiplication that 
\begin{equation}\label{four}S^T \widetilde{A}_j S=\widetilde{D}_j,\ \hbox{for}\ 1\leq j\leq m.
\end{equation} 
Then for reduced 
$r \times r$ matrix pencils (and $\l=(\l_1,\ldots,\l_m) \in \mathbb{C}^m$)
$$\widetilde{A}(\lambda) = \sum_{j=1}^m \lambda_j \widetilde{A}_j\ \hbox{and}\ \widetilde{D}(\lambda) = \sum_{j=1}^m \lambda_j \widetilde{D}_j$$
  (\ref{four}) gives
\begin{equation}\label{five}S^T \widetilde{A}(\lambda) S=\widetilde{D}(\lambda),\end{equation} and, in particular,
\begin{equation}\label{six}S^T \widetilde{A}(\lambda_0) S=\widetilde{D}(\lambda_0).
\end{equation} 
Since $\widetilde{A}(\lambda_0)$ and $\widetilde{D}(\lambda_0)$ are invertible, it follows that 
$S$ is invertible and 
combining (\ref{four}) and  (\ref{six}) gives 
$${\widetilde{D}(\lambda_0)}^{-1} \widetilde{D}_j = S^{-1} {\widetilde{A}(\lambda_0)}^{-1}{(S^T)}^{-1}S^T \widetilde{A}_j S =  S^{-1} {\widetilde{A}(\lambda_0)}^{-1} \widetilde{A}_j S \,,\ 1\leq j\leq m.$$
In particular, $ S^{-1} {\widetilde{A}(\lambda_0)}^{-1} \widetilde{A}_j S$ are diagonal for all $j=1, \ldots, m$. 
In other words the $r\times r$ matrices 
$$L_j ={\widetilde{A}(\lambda_0)}^{-1} \widetilde{A_j},\ \hbox{for}\ \ 1\leq j\leq m$$ are SDS and we are done.\\
\smallskip

For the opposite direction, let us assume that $\dim(\bigcap_{j=1}^m \mathrm{ker} A_j)=n-r$ and that $L_1, \ldots, L_m\ \text{ are SDS}$. 
Then from Lemma \ref{kernels2} there exists $Q \in \GL_n$ such that
$$Q^T A_j Q = \widetilde{A}_j \oplus 0_{n-r}, \qquad j=1, \ldots, m,$$
with $\widetilde{A}_j \in \MS_r$ and $\widetilde{A}(\l_0) \in \GL_r \cap \MS_r$. Construct $L_j = {\widetilde{A}(\lambda_0)}^{-1} \widetilde{A}_j, \quad j=1, \ldots, m$ as in Definition \ref{reducedrank} above. By hypothesis, these matrices are SDS so from Theorem  \ref{mainfullrank} it follows that $\widetilde{A}_1, \ldots \widetilde{A}_m$ are SDC, namely, there exists $P\in \GL_r$ such that 
$P^T \widetilde{A}_j P = D_j$, for $D_j$ diagonal in $\M_r$, for all $j=1, \ldots, m$. Define $R := P\oplus I_{n-r}
 \in \GL_n$. Then 
 $$R^T (Q^T A_j Q) R = (P^T \widetilde{A}_j P) \oplus 0_{n-r} = D_j \oplus 0_{n-r}, \quad j=1, \ldots, m.$$
Thus, for $S = QR \in \GL_n$ we have
$$S^T A_j S = D_j \oplus 0_{n-r}\,,$$ 
diagonal for all $j=1, \ldots, m$. Thus, $A_1, \ldots, A_m$ are SDC.
\end{proof}
~

\subsection{A procedure to solve the SDC problem}
~\\

The above results allow us now to determine in a finite number of steps whether or not a set of matrices are SDC. Given $A_1, \ldots, A_m \in \MS_n$, let  $r:=\mathrm{rank}\big(A_1,\ldots,A_m\big)$. From (\ref{kerdim}), we have $$\dim\big(\bigcap_{j=1}^m \mathrm{ker} A_j\big) \leq n-r$$ and  Theorems  \ref{main} and \ref{well} now give us the following procedure.
\medskip

\begin{enumerate}
\item  If $\ \dim(\bigcap_{j=1}^m \mathrm{ker} A_j)<n-r$ then $A_1, \ldots, A_m$ are not SDC. 

\item   If $\ \dim(\bigcap_{j=1}^m \mathrm{ker} A_j)=n-r$ then we calculate $L_1, \ldots, L_m$ from Eq.~(\ref{eq:Ldef}), for some $\l_0 \in \mathbb{C}^m$ with $r=\mathrm{rank} A(\l_0)$. If $L_1, \ldots, L_m$ do not pairwise commute then $A_1, \ldots, A_m$ are not SDC. 

\item  If $\ \dim(\bigcap_{j=1}^m \mathrm{ker} A_j)=n-r$, $L_1, \ldots, L_m$ do pairwise commute, and if each $L_1, \ldots, L_m$ is diagonalizable by similarity then   $A_1, \ldots, A_m$  are SDC. Otherwise (that is, if any one $L_j$ is not diagonalizable by similarity)  $A_1, \ldots, A_m$ are not SDC. 
\end{enumerate}
\medskip
\begin{remark}
 Regarding (1) above, we note that estimation may be sufficient to determine if $\ \dim(\bigcap_{j=1}^m \mathrm{ker} A_j)<n-r$, as in Example \ref{second} below. Regarding (2), since $\sum_{j=1}^m (\lambda_0)_j L_j = I_n$, it suffices to check if $m-1$ of $L_1, \ldots, L_m$ pairwise commute. Regarding (3), we recall that $L_i$ are not symmetric in general.
\end{remark}

\begin{example}[$n=2,m=2$]\label{complexneeded} Let    
$A_1 = \left(
\begin{array}
[c]{cc}%
0 & 1\\
1 & 1
\end{array}
\right), \,\, A_2 = \left(
\begin{array}
[c]{cc}%
1 & 1\\
1 & 0
\end{array}
\right)\,.$ We apply the above procedure.
\begin{enumerate}
\item $\mathrm{ker}{A_1} = \mathrm{ker}{A_2} = \{0\}$ so 
 $\dim(\bigcap_{j=1}^2\mathrm{ker} A_j)=0.$
As $A_1$ is nonsingular we take $\lambda_0 = (1,0)$ so $A(\lambda_0) = A_1$ and  $\mathrm{rank}{A_1} = 2$. Therefore $r=2$ and $\ \dim(\bigcap_{j=1}^m \mathrm{ker} A_j)=n-r$ holds and we continue to the next step.
\item We compute $L_1 = A_1^{-1} A_1 = I_2$ and $L_2 = A_1^{-1} A_2 = \left(
\begin{array}
[c]{cc}%
0 & -1\\
1 & 1
\end{array}
\right)$. These matrices (trivially) commute so we continue to the next step.
\item $L_1$ is diagonal. $L_2$ is diagonalizable by similarity as it has $2$ different eigenvalues: $d_{\pm} = (1\pm i \sqrt{3})/2$. Therefore the matrices $A_1, A_2$ are SDC. 
\end{enumerate}

Explicitly, $P^{-1} L_2 P = \mathrm{\mathrm{diag}}(d_+, d_-)$ with $P = \left(
\begin{array}
[c]{cc}%
 a d_- & b d_+ \\
- a & -b
\end{array}
\right)$
for any $a, b \in \mathbb{C}$ with $ab \neq 0$. Note that $P$ cannot be made real by any choice of the constants $a,b$. We have, finally, 
$$P^T A_1 P = i \sqrt{3} \, \mathrm{\mathrm{diag}}(a^2, -b^2), \qquad P^T A_2 P = i \sqrt{3} \, \mathrm{\mathrm{diag}}(a^2 d_+, -b^2 d_-).$$

\end{example}

\begin{example}[$n=3,m=2$]\label{second}
Let    
$A_1 = \left(
\begin{array}
[c]{ccc}%
1 & 1 & 0\\
1 & 0 & 0\\
0 & 0 & 0
\end{array}
\right) , \,\,\, A_2 = \left(
\begin{array}
[c]{ccc}%
0 & 0 & 1\\
0 & 0 & 0\\
1 & 0 & 0
\end{array}
\right) \,.$ 

We apply the above procedure.
\begin{enumerate}
\item We calculate $\mathrm{ker} A_1 = \mathrm{span}\{(0,0,1)^T\}$ and $\mathrm{ker} A_2 = \mathrm{span}\{(0,1,0)^T\}$.  Thus $\mathrm{ker}{A_1} \cap \mathrm{ker}{A_2} = \{0\}$ so
$\dim(\bigcap_{j=1}^2\mathrm{ker} A_j)=0.$

Since
$$\det A(\lambda) =  \det \left(
\begin{array}
[c]{ccc}%
\lambda_1 & \lambda_1& \lambda_2\\
\lambda_1 & 0 & 0\\
\lambda_2 & 0 & 0
\end{array}
\right) = 0 \quad \text{for all} \quad \lambda \in \mathbb{C}^{2}\,,$$
we have  $r\leq 2$ and then $n-r\geq 1$. Therefore $\ \dim(\bigcap_{j=1}^2 \mathrm{ker} A_j)<n-r$ and hence $A_1, A_2$ are not SDC.  
\end{enumerate}

\end{example}

\section{Discussion}
\label{sec:4}

In this paper we solved the long-standing problem of simultaneous diagonalization via congruence in the complex {}{symmetric} case, providing also an explicit set of steps to solve this problem.  {The complex case has  applications in signal processing, in particular to the problem of blind source separation. This latter problem is based on the exact SDC problem but, due to experimental and numerical errors in obtaining the target matrices $A_1, \ldots, A_m$, it relies on the so-called approximate joint diagonalization, which is an optimisation problem. Our results could shed light on these approximate problems, as these problems usually consider an ad-hoc cost function \cite{belouchrani1997blind, yeredor2002non}, which does not take into account the kernels of the target matrices.}

Some optimisation-related applications consider the special case where, in the context of our Definition \ref{def:1}, the symmetric matrices $A_1, \ldots, A_m$ are real, and the corresponding transformation matrix $P$ and resulting diagonal matrices $D_1, \ldots, D_m$ are required to be real. In the context of Theorem \ref{main} above, such a case would impose extra conditions of realness on the eigenvectors and eigenvalues of the reduced matrices $L_1, \ldots, L_m$. 

In many applications in genetics the matrices $L_1, \ldots, L_m$ turn out to commute, but may not necessarily be diagonalizable. Thus, the SDC problem could be relaxed to a weaker problem, namely that of simultaneous block diagonalization \cite{Uhlig73}. 

{Further research on building an algorithm to solve the SDC problem will focus on developing an efficient method for finding $\lambda_0$ such that the pencil $A(\lambda_0)$ has maximum rank.}
\medskip

{\bf Acknowledgments}
{\it The authors thank the referees for several helpful comments and suggestions and for bringing key references to our attention.}

\end{document}